\documentclass[12pt]{article}

\usepackage[hmargin=1.3in,vmargin=1.3in]{geometry}

\usepackage{amsfonts,amssymb,amsmath,amsthm,bm}

\numberwithin{equation}{section}

\RequirePackage[colorlinks,pagebackref,linkcolor=blue,filecolor = blue,citecolor = blue, urlcolor =blue]{hyperref}

\newtheorem{lemma}{Lemma}[section]
\newtheorem{theorem}[lemma]{Theorem}
\newtheorem{prop}[lemma]{Proposition}
\newtheorem{cor}[lemma]{Corollary}

\newtheorem{defn}{Definition}

\theoremstyle{remark}

\begin{document}

\title{Subset Sums of Quadratic Residues over Finite Fields \thanks{Research  is supported
  in part by 973 Program (2013CB834203),  National Natural Science Foundation of China under Grant No.61202437 and 11471162, in part by Natural Science Basic Research Plan in Shaanxi Province of China under Grant No.2015JM1022 and Natural Science Foundation of the Jiangsu Higher Education Institutes of China under Grant No.13KJB110016.} }
\author{\small Weiqiong Wang$^{a}$\ \ Li-Ping Wang$^{b\dag}$\ \ Haiyan Zhou$^c$\\ \small $a.$ School of Science,
Chan'an University, Xi'an 710064, China \\
\small Email: wqwang@chd.edu.cn\\
\small $b.$ Institute of Information Engineering, Chinese Academy of Sciences
Beijing 100093, China\\
\small Email: wangliping@iie.ac.cn\\
\small $c.$ School of Mathematics, Nanjing Normal University, Nanjing 210023, China\\ \small Email: haiyanxiaodong@gmail.com}
\date{}
\maketitle
\thispagestyle{empty}
\begin{abstract}

In this paper, we derive an explicit combinatorial formula for the number of $k$-subset sums of quadratic residues over
 finite fields.
\end{abstract}

\section{Introduction}
Let $\mathbb{F}_q$ be the finite field with $q=p^s$ elements, where $p$ is a prime and $s\geq 1$ is an integer. Let $H$ be a subset of $\mathbb{F}_{q}$, and $k(1\leq k\leq |H|)$ be a positive integer. For $b\in \mathbb{F}_{q}$, let $N_H(k,b)$ denotes the number of  $k$-element subsets $S\subseteq H$ such that
\begin{equation}
\sum_{a\in S}a=b.
\end{equation}
Understanding the number $N_H(k,b)$
is the well known $k$-subset sum problem over finite fields. It arises from several applications in coding theory, cryptography, graph theory and some other fields. For example, it is directly
related to the deep hole problem of generalized Reed-Solomon codes with evaluation set $H$ \cite{ChenMur07,ChenWan07,ChenWan10,GGG15}.
It is also related to the spectrum and the diameter of the Wenger type graphs \cite{CaoLu15}.

 However, the $k$-subset sum problem over finite fields for general $H$ is well known to be NP-hard.  The difficulty mainly comes from the combinatorial flexibility of choosing the subset $H$ and also the lack of algebraic structure of $H$.
 Due to the NP-hardness, there is little that we can say about the  exact value of $N_H(k,b)$ in general. But  from mathematic point of view, we would like to obtain an explicit formula or at least an asymptotic formula for $N_H(k,b)$.  This is again out of our expectation in general. But if $H$ is certain special subset with good algebraic structure,
 one can hope to obtain the exact value or asymptotic formula for $N_H(k,b)$. For example, it is known that if $\mathbb{F}_{q}-H$ is a small set, there is a simple asymptotic formula for  $N_H(k,b)$\cite{LiW08}. In addition, if $H=\mathbb{F}_{q}$, or $\mathbb{F}_{q}^*$, or any additive subgroup of $\mathbb{F}_{q}$, there is also an explicit combinatorial
 formula for  $N_H(k,b)$\cite{LiW08,LiW10,LiWan12}.

If $H$ is a multiplicative subgroup of $\mathbb{F}_{q}$ of index $m$ (thus $m$ divides $q-1$), the subset sum problem becomes harder as it is a non-linear algebraic problem with many combinatorial constraints. Zhu and Wan \cite{ZhuW12} provided  an asymptotic formula for $N_H(k,b)$ in this case.
As a consequence, they proved that for small index $m=[\mathbb{F}_{q}^*:H]$ and $6\ln q<k<\frac{q-1}{2m}$, $N_H(k,b)>0$ for all $b\in \mathbb{F}_{q}$. This is
the only known result in the case that $H$ is a proper multiplicative subgroup.

The complexity of the subset sum problem grows as the index $m$ of the subgroup $H$ grows. In the simplest case $m=1$, then $H= \mathbb{F}_q^*$ and an explicit
combinatorial formula for $N_H(k,b)$ is known. In this paper, we study the next simplest case $m=2$. Our main result is an
explicit combinatorial formula for $N_H(k,b)$,  where $H$ is the subgroup of quadratic residues in $\mathbb{F}_q^*$, that is, $H=\{x^2\mid x\in \mathbb{F}_{q}^*\}$.
Equivalently, we obtain an explicit combinatorial formula for
  \begin{equation}
  N_H(k,b)=\frac{1}{k!}\sharp \{(y_1, y_2, \cdots, y_k)\in H^k\mid y_1+y_2+\cdots +y_k=b, y_i\neq y_j\textrm{ for }\forall \ i\neq j\}.
\end{equation}
Note that there is the coefficient $\frac{1}{k!}$ because $N_H(k,b)$ denotes the number of the unordered $k$-tuples with distinct coordinates satisfying the equation $ y_1+y_2+\cdots +y_k=b$ with $y_i\in H$. When $m\geq 3$, one should not expect an explicit formula for $N_H(k,b)$.

Our main tools in this paper are the new sieve \cite{LiW10}, some combinatorial properties and the standard character sums over finite fields. Our technique is to find out the exact number of points with nonzero coordinates on quadratic diagonal equations  first, and then sieve twice to obtain our desired results. Our formula is more complicated in the
case that $s=[\mathbb{F}_q :\mathbb{F}_p]$ is odd, but greatly simplified in the case $s= [\mathbb{F}_q :\mathbb{F}_p]$ is even.

\section{Preliminary}

In this section, we review some basic properties of Gauss-Jacobi sums that will be used in the following sections.

 A multiplicative character on $\mathbb{F}_{q}^{*}$ is a map $\chi$ from $\mathbb{F}_{q}^{*}$ to the nonzero complex numbers set $\mathbb{C}^*$ which satisfies $\chi(xy)=\chi(x)\chi(y)$ for all $x,y\in \mathbb{F}_{q}^{*}$. We extend the definition to the whole field $\mathbb{F}_{q}$ by defining
 $$
 \chi(0)=\left\{
\begin{array}{ll}1 &\textrm{ if $\chi = 1$,} \\  0 &\textrm{ otherwise.}
\end{array}\right.
 $$
\begin{defn}
  Let $\chi$ be a multiplicative character on $\mathbb{F}_{q}$ and $a \in \mathbb{F}_{q}$. Set
  $$
   G_a(\chi)=\sum_{t\in \mathbb{F}_{q}}\chi(t)\zeta^{Tr(at)},
  $$
where $\zeta=e^{\frac{2\pi i}{p}}$ and Tr denotes the trace from $\mathbb{F}_{q}$ to $\mathbb{F}_{p}$. We call $G_a(\chi)$ the Gauss sum on  $\mathbb{F}_{q}$, and usually denote $G_1(\chi)$
by $G(\chi)$.
\end{defn}

\begin{prop}\cite{LidlN}
Let $\mathbb{F}_{q}$ be a finite field with $q=p^s$, where $p$ is an odd prime and $s\in \mathbb{N}$. Let $\chi$ be the quadratic character of $\mathbb{F}_{q}$. Then we have
\begin{equation}
G(\chi)=\left\{
\begin{array}{ll}(-1)^{s-1}q^{1/2} &\textrm{ if $p\equiv 1($mod $4)$}, \\  (-1)^{s-1}i^s q^{1/2} &\textrm{ if $p\equiv 3($mod $4)$},
\end{array}\right.
\end{equation}
where $i=\sqrt{-1}.$
\end{prop}
\begin{defn}
Let $\chi_1,\cdots , \chi_{n}$ be multiplicative characters on $\mathbb{F}_{q}$. We define the following Jacobi type sum by
$$J(\chi_1,\cdots , \chi_{n})=\sum_{y_1+y_2+\cdots + y_n =1}\chi_1(y_1)\cdots \chi_n(y_n).$$
$$J^*(\chi_1,\cdots , \chi_{n})=\mathop{\sum_{y_1+y_2+\cdots + y_n =1}}_{{y_i\neq 0}}\chi_1(y_1)\cdots \chi_n(y_n).$$
$$J_0(\chi_1,\cdots , \chi_{n})=\sum_{y_1+y_2+\cdots + y_n =0}\chi_1(y_1)\cdots \chi_n(y_n).$$
$$J_0^*(\chi_1,\cdots , \chi_{n})=\mathop{\sum_{y_1+y_2+\cdots + y_n =0}}_{{y_i\neq 0}}\chi_1(y_1)\cdots \chi_n(y_n).$$
\end{defn}

These Jacobi type sums have the following properties.
\begin{prop}\cite{ZhuW12}
If $\chi_{i_1}=\cdots=\chi_{i_e}=1$, and  $\chi_{i_{e+1}}\neq 1,\cdots, \chi_{i_{n}}\neq 1$. Then
\begin{equation}
J_0(\chi_1,\cdots , \chi_{n})=\left\{
\begin{array}{ll}q^{n-1} &\textrm{ if $e=n $}, \\  0 &\textrm{ if $1\leq e< n$}, \\ 0 &\textrm{ if $e=0, \chi_1\cdots\chi_{n}\neq 1$}, \\
\chi_n(-1)(q-1)J(\chi_1,\cdots , \chi_{n-1}) &\textrm{otherwise}.
\end{array}\right.
\end{equation}
\begin{equation}
J^*(\chi_1,\cdots , \chi_{n})=\left\{
\begin{array}{ll}\frac{1}{q}[(q-1)^{n}-(-1)^n] &\textrm{ if $e=n $}, \\  (-1)^eJ(\chi_{i_{e+1}},\cdots, \chi_{i_{n}}) &\textrm{ if $0\leq e< n$}.
\end{array}\right.
\end{equation}
\begin{equation}
J_0^*(\chi_1,\cdots , \chi_{n})=\left\{
\begin{array}{ll}\frac{1}{q}[(q-1)^{n}+(q-1)(-1)^{n}] &\textrm{ if $e=n $}, \\  (-1)^eJ_0(\chi_{i_{e+1}},\cdots, \chi_{i_{n}}) &\textrm{ if $0\leq e< n$}.
\end{array}\right.
\end{equation}
\end{prop}
Based on the above two propositions, we can derive the following conclusion.
\begin{lemma}
Let $\mathbb{F}_{q}$ be a finite field with $q=p^s$, where $p$ is an odd prime and $s\in \mathbb{N}$. Let $e\in \mathbb{N}$, and  $\chi$ be the nontrivial quadratic character of $\mathbb{F}_{q}$.\\
$(1)$ If either  $p\equiv 1($mod $4)$, or $p\equiv 3($mod $4)$ and $s$ is even, then
$$
J(\underbrace{\chi,\cdots , \chi}_{e})=\left\{
\begin{array}{ll}q^{\frac{e-1}{2}} &\textrm{ $e$ odd}, \\  -q^{\frac{e}{2}-1} &\textrm{  $e$ even}.
\end{array}\right.$$
$(2)$ If $p\equiv 3($mod $4)$ and $s$ is odd, then
$$J(\underbrace{\chi,\cdots , \chi}_{e})=\left\{
\begin{array}{ll}(-q)^{\frac{e-1}{2}} &\textrm{$e$ odd }, \\  (-q)^{\frac{e}{2}-1} &\textrm{ $e$ even}.
\end{array}\right.$$

\end{lemma}

\begin{proof}
 Based on proposition $2.1$, and the relationships between Gauss sum and Jacobi sum, we can prove that if $e$ is odd,
$$
J(\underbrace{\chi,\cdots , \chi}_{e})=\frac{G(\chi)^e}{G(\chi)}=\left(G(\chi)^2\right)^{\frac{e-1}{2}}=\left\{
\begin{array}{ll}q^{\frac{e-1}{2}} &\textrm{ if $p\equiv 1($mod $4)$}, \\  (-1)^{\frac{s(e-1)}{2}}q^{\frac{e-1}{2}} &\textrm{ if $p\equiv 3($mod $4)$}.
\end{array}\right.$$

On the other hand, if $e$ is even, we know from Theorem $5.21$ in \cite{LidlN} that $$J(\underbrace{\chi,\cdots , \chi}_{e})=-\chi(-1)\cdot J(\underbrace{\chi,\cdots , \chi}_{e-1}).$$
 Since $\chi(-1)=1$ for $p\equiv 1($mod $4)$ and $\chi(-1)=(-1)^s$ for $p\equiv 3($mod $4)$, it follows that in this case
$$J(\underbrace{\chi,\cdots , \chi}_{e})=\left\{
\begin{array}{ll}-q^{\frac{e}{2}-1} &\textrm{ if $p\equiv 1($mod $4)$}, \\  -(-1)^{\frac{se}{2}}q^{\frac{e}{2}-1} &\textrm{ if $p\equiv 3($mod $4)$}.
\end{array}\right.$$

Finally we arrive at the desired results by discussing on $s$ and $e$.
\end{proof}

\section{Counting points on quadratic equations}

A diagonal equation over $\mathbb{F}_{q}$ is an equation of the form
\begin{equation}
a_1 x_1^{k_1}+\cdots + a_n x_n^{k_n}=b
\end{equation}
with positive integers $k_1,\cdots, k_n$, and $a_1,\cdots, a_n\in \mathbb{F}^*_{q}$, $b\in \mathbb{F}_{q}$.
The number of solutions in $\mathbb{F}_q$ of this kind of diagonal equations can be expressed by Jacobi type sums, and the precise number of solutions can also be obtained when $k_1=k_2=\cdots=k_n=2$\cite{LidlN}. However,
sometimes we only want solutions in $\mathbb{F}_{q}^*$. So in this paper, we first calculate the number of solutions with nonzero coordinates of quadratic diagonal equations. We denote it by
\begin{equation}
     N^{*}(n,b)=\sharp\{\left(x_1,x_2,\cdots,x_{n}\right)\in (\mathbb{F}_{q}^{*})^{n}\mid a_1 x_1^{2}+a_2 x_2^{2}+\cdots +a_n x_n^{2}=b\},
\end{equation}
where $a_1,\cdots, a_n\in \mathbb{F}^*_{q}$, and $b\in \mathbb{F}_{q}$.

Obviously, the equation in $(3.2)$  reduces to a  linear equation if the characteristic $p$ of $\mathbb{F}_{q}$ is $2$ since it is a square. So in the following sections, we assume the characteristic $p$ is odd, which is all we need since $2$ divides $q-1$ in our applications. We provide firstly some lemmas that will be used in the following theorems.

\begin{lemma} Let $n$ and $m$ be two positive integers,  $m<n$. Let $f(x)=\left(1-\sqrt{q}x\right)^m\cdot\left(1+\sqrt{q}x\right)^{n-m}$. Then we have the following results:\\
\begin{eqnarray*}
& &\mathop{\sum_{e=1}^n}_{\textrm{$e$ even}}\sum_{j=0}^m(-1)^{j}\binom{m}{j}\binom{n-m}{e-j}(\sqrt{q})^e
=\frac{1}{2}\left[f(x)+f(-x)\right]\mid_{x=1}-1;\\
& &\mathop{\sum_{e=1}^n}_{\textrm{$e$ odd}}\sum_{j=0}^m(-1)^{j}\binom{m}{j}\binom{n-m}{e-j}(\sqrt{q} )^e
=\frac{1}{2}\left[f(x)-f(-x)\right]\mid_{x=1};\\
& &\mathop{\sum_{e=1}^n}_{4\mid e}\sum_{j=0}^m(-1)^{j}\binom{m}{j}\binom{n-m}{e-j}(\sqrt{q} )^e
=\frac{1}{4}\left[f(x)+f(-x)+f(-ix)+f(ix)\right]\mid_{x=1}-1;\\
& &\mathop{\sum_{e=1}^n}_{\textrm{$e\equiv 1$ (mod $4)$}}\sum_{j=0}^m(-1)^{j}\binom{m}{j}\binom{n-m}{e-j}(\sqrt{q} )^e
=\frac{1}{4}\left[f(x)-f(-x)+if(-ix)-if(ix)\right]\mid_{x=1};\\
& &\mathop{\sum_{e=1}^n}_{\textrm{$e\equiv 2$ (mod $4)$}}\sum_{j=0}^m(-1)^{j}\binom{m}{j}\binom{n-m}{e-j}(\sqrt{q} )^e
=\frac{1}{4}\left[f(x)+f(-x)-f(-ix)-f(ix)\right]\mid_{x=1};\\
& &\mathop{\sum_{e=1}^n}_{\textrm{$e\equiv 3$ (mod $4)$}}\sum_{j=0}^m(-1)^{j}\binom{m}{j}\binom{n-m}{e-j}(\sqrt{q} )^e
=\frac{1}{4}\left[f(x)-f(-x)-if(-ix)+if(ix)\right]\mid_{x=1},
\end{eqnarray*}
where $i=\sqrt{-1}$.
\end{lemma}
\begin{proof}
We only prove the first two equations, the other equations can be obtained with the same method.
\begin{eqnarray*}
 f(x)&= &  (1-\sqrt q x )^m\cdot(1+\sqrt qx )^{n-m}\\
 &=&\sum_{j=0}^m (-1)^j\binom{m}{j}(\sqrt qx)^j\cdot\sum_{k=0}^{n-m}\binom{n-m}{k}(\sqrt qx )^k  \\
&=&\sum_{e=0}^n\sum_{j=0}^m(-1)^j\binom{m}{j}\binom{n-m}{e-j}(\sqrt q x )^e\\
\end{eqnarray*}
Similarly, we have
$$f(-x)=-\mathop{\sum_{e=0}^n}_{{e\  \textrm{odd} }}\sum_{j=0}^m(-1)^j\binom{m}{j}\binom{n-m}{e-j}(\sqrt qx )^e+\mathop{\sum_{e=0}^n}_{{e\  \textrm{even} }}\sum_{j=0}^m(-1)^j\binom{m}{j}\binom{n-m}{e-j}(\sqrt q x)^e.$$

We can easily obtain the first two equations by combing the above two results.

\end{proof}
\begin{lemma}
 Let $\chi$ be the nontrivial quadratic character of $\mathbb{F}_{q}$. For $a_1,a_2,\cdots,a_n\in \mathbb{F}_{q}^*$, set $m=\sharp\{1\leq i \leq n\mid \chi(a_i)=1\}$,
 then we have for $e\geq 1$,
\begin{equation}
\sum_{1\leq i_1<\cdots <i_e\leq n}\chi(a_{i_1}\cdots a_{i_e})=(-1)^e\sum_{i=0}^m(-1)^i\binom{m}{i}\binom{n-m}{e-i}.
\end{equation}

\end{lemma}
\begin{proof}
It is not difficulty to prove that $\sum_{1\leq i_1<\cdots <i_e\leq n}
\chi(a_{i_1}\cdots a_{i_e})$ is just the coefficient of $x^e$ in the polynomial $\prod_{i=1}^n\left[1+\chi(a_{i})x\right]$, that is to say, the coefficient of $x^e$ in the polynomial $(1+x)^m(1-x)^{n-m}$ since  $m=\sharp\{1\leq i \leq n\mid \chi(a_i)=1\}$.

On the other hand,
$$(1+x)^m(1-x)^{n-m}=\sum_{i=0}^m\binom{m}{i}x^i\sum_{j=0}^{n-m}(-1)^j\binom{n-m}{j}x^j=\sum_{e=0}^n(-1)^e\sum_{i=0}^m(-1)^i\binom{m}{i}\binom{n-m}{e-i}x^e.$$
 Comparing the coefficient of $x^e$ yields the desired fact.
 \end{proof}

\begin{theorem}
Let $\mathbb{F}_{q}$ be the finite field with $q=p^s$. Let $\chi$ be the nontrivial quadratic character of $\mathbb{F}_{q}$. For all $b$ in $\mathbb{F}_{q}^{*}$, set $m=\sharp\{1\leq i\leq n\mid \chi(a_i)=\chi(b)\}$.\\
$(1)$ If either $p\equiv 1($mod $4)$, or $p\equiv 3($mod $4)$ and $s$ is even, then
$$ N^{*}(n,b)=\frac{(q-1)^n}{q}-\frac{(-1)^n}{2q}\left[\left(1-\sqrt{q}\right)^{m+1}\left(1+\sqrt{q}\right)^{n-m}+
 \left(1+\sqrt{q}\right)^{m+1}\left(1-\sqrt{q}\right)^{n-m}\right].$$
$(2)$ If $p\equiv 3($mod $4)$, $s$ is odd, then
\begin{eqnarray*}
N^{*}(n,b)=\frac{(q-1)^n}{q}-\frac{(-1)^n}{2q}\left[\left(1-\sqrt{q}i\right)^{m}\left(1+\sqrt{q}i\right)^{n-m+1}+ \left(1+\sqrt{q}i\right)^{m}\left(1-\sqrt{q}i\right)^{n-m+1}\right],
\end{eqnarray*}
where $i=\sqrt{-1}$.
\end{theorem}

\begin{proof}
Firstly, we consider the case $p\equiv 1($mod $4)$, and the case of $p\equiv 3($mod $4)$ and $s$ is even. Without loss of generality, we can assume $b=1$. In this case, $m=\sharp\{1\leq i \leq n\mid \chi(a_i)=1\}$. The following proof mainly based on Proposition $2.2$, Lemma $2.3$, Lemma $3.1$ and $3.2$.
\begin{eqnarray*}
   & &N^{*}(n,1)=\mathop{\sum_{y_1+y_2+\cdots + y_n =1}}_{{y_i\neq 0}}\prod_{i=1}^n \sharp \{a_i x_i^2=y_i\}     = \mathop{\sum_{y_1+y_2+\cdots + y_n =1}}_{{y_i\neq 0}}\prod_{i=1}^n \sum_{\chi_i^2=1}\chi_i \left(\frac{y_i}{a_i}\right)   \ \ \ \ \ \ \ \ \\
&=&\sum_{\chi_1^2=\chi_2^2=\cdots=\chi_n^2=1}\prod_{i=1}^n \chi_i^{-1}(a_i)\mathop{\sum_{y_1+y_2+\cdots + y_n =1}}_{{y_i\neq 0}} \chi_1(y_1)\cdots \chi_n(y_n)  \\
&=&\frac{1}{q}\left[(q-1)^n-(-1)^n\right]+\sum_{e=0}^{n-1}\sum_{\mbox{\tiny$\begin{array}{c}
\chi_1^2=\chi_2^2=\cdots=\chi_n^2=1\\
\chi_{i_1}=\cdots =\chi_{i_e}=1\\
\chi_{i_{e+1}}=\cdots =\chi_{i_n}\neq 1\end{array}$}}\prod_{i=1}^n\chi_i^{-1}(a_i)J^*(\chi_1,\chi_2,\cdots,\chi_n)   \\
&=& \frac{1}{q}\left[(q-1)^n-(-1)^n\right]+\sum_{e=0}^{n-1}\sum_{\mbox{\tiny$\begin{array}{c}
\chi_1^2=\chi_2^2=\cdots=\chi_n^2=1\\
\chi_{i_1}=\cdots =\chi_{i_e}=1\\
\chi_{i_{e+1}}=\cdots =\chi_{i_n}\neq 1\end{array}$}}
\prod_{i=1}^n\chi_i^{-1}(a_i)(-1)^e J\left(\chi_{i_{e+1}},\cdots , \chi_{i_n}\right)\\
&=& \frac{1}{q}\left[(q-1)^n-(-1)^n\right]+(-1)^n\sum_{e=1}^{n}(-1)^e\sum_{1\leq i_1<\cdots <i_e\leq n}
\chi(a_{i_1}\cdots a_{i_e}) J(\underbrace{\chi,\cdots , \chi}_{e})\ \\
&=& \frac{1}{q}\left[(q-1)^n-(-1)^n\right]+(-1)^n\sum_{e=1}^{n}\sum_{i=0}^m(-1)^i\binom{m}{i}\binom{n-m}{e-i}J(\underbrace{\chi,\cdots , \chi}_{e})\\
 &=&\frac{1}{q}\left[(q-1)^n-(-1)^n\right]+(-1)^n\mathop{\sum_{e=1}^n}_{{e\  \textrm{odd} }}\sum_{i=0}^m(-1)^i\binom{m}{i}\binom{n-m}{e-i}(\sqrt q )^{e-1}\\
 & &-(-1)^n\mathop{\sum_{e=1}^n}_{{e\  \textrm{even} }}\sum_{i=0}^m(-1)^i\binom{m}{i}\binom{n-m}{e-i}(\sqrt q )^{e-2}\\
 &=&\frac{1}{q}\left[(q-1)^n-(-1)^n\right]-\frac{(-1)^n}{2q}\{\left(1-\sqrt{q}\right)^{m+1}\left(1+\sqrt{q}\right)^{n-m}\\
 &&+ \left(1+\sqrt{q}\right)^{m+1}\left(1-\sqrt{q}\right)^{n-m}-2\}\\
 &=&\frac{(q-1)^n}{q}-\frac{(-1)^n}{2q}\left[\left(1-\sqrt{q}\right)^{m+1}\left(1+\sqrt{q}\right)^{n-m}+
 \left(1+\sqrt{q}\right)^{m+1}\left(1-\sqrt{q}\right)^{n-m}\right]
\end{eqnarray*}

Generally, if $b\neq 1$, we can transform the equation $a_1 x_1^{2}+a_2 x_2^{2}+\cdots +a_n x_n^{2}=b$ into $a_1 b^{-1} x_1^{2}+a_2 b^{-1}  x_2^{2}+\cdots +a_n b^{-1} x_n^{2}=1$. In this case, we can derive the same formula but $m=\sharp\{1\leq i \leq n\mid \chi(a_i)=\chi(b)\}$.

The case of $p\equiv 3($mod $4)$ and $s$ is odd can be similarly proved.
\end{proof}

Similarly, we can solve out the number of solutions  with  nonzero coordinates of the equation in $3.2$ when $b=0$.

 \begin{lemma}
  Let $\mathbb{F}_{q}$ be a finite field with $q=p^s$. Let $\chi$ be a nontrivial quadratic character of $\mathbb{F}_q$ and  $e$  be a positive even integer. Then
\begin{eqnarray*}
& &J_0(\underbrace{\chi,\cdots , \chi}_{e})=\chi(-1)(q-1)J(\underbrace{\chi,\cdots , \chi}_{e-1})\\
&=&\left\{
\begin{array}{ll}(q-1)q^{\frac{e}{2}-1} &\textrm{ if either $p\equiv 1($mod $4)$, or $p\equiv 3($mod $4)$ and $s$ is even}, \\  -(q-1)(-q)^{\frac{e}{2}-1} &\textrm{ if $p\equiv 3($mod $4)$ and $s$ is odd.}
\end{array}\right.
\end{eqnarray*}
 \end{lemma}

 This Lemma follows from  Theorem $5.20$ in \cite{LidlN} and Lemma $2.3$ in this paper.

 \begin{theorem}Let $\mathbb{F}_{q}$ be a finite field with $q=p^s$. Let $\chi$ be the nontrivial quadratic character of $\mathbb{F}_{q}$ and  set $m=\sharp\{1\leq i\leq n\mid \chi(a_i)=1\}$.\\
$(1)$ If either $p\equiv 1($mod $4)$, or $p\equiv 3($mod $4)$ and $s$ is even, then
\begin{eqnarray*}
N^*(n,0)=\frac{(q-1)^n}{q}+(-1)^n\frac{q-1}{2q}\left[(1-\sqrt{q})^m(1+\sqrt{q})^{n-m}+(1+\sqrt{q})^m(1-\sqrt{q})^{n-m}\right].
 \end{eqnarray*}
$(2)$ If $p\equiv 3($mod $4)$, $s$ is odd, then\\
\begin{eqnarray*}
N^*(n,0)=\frac{(q-1)^n}{q}+(-1)^n\frac{q-1}{2q}\left[\left(1-\sqrt{q}i\right)^{m}\left(1+\sqrt{q}i\right)^{n-m}+ \left(1+\sqrt{q}i\right)^{m}\left(1-\sqrt{q}i\right)^{n-m}\right].
\end{eqnarray*}
where $i=\sqrt{-1}$.
\end{theorem}
\begin{proof}If either $p\equiv 1($mod $4)$, or $p\equiv 3($mod $4)$ and $s$ is even, we can provide the following proof based on the Proposition $2.2$, Lemma $2.3$,  $3.2$ and $3.4$.
 \begin{eqnarray*}
&& N^{*}(n,0)=\mathop{\sum_{y_1+y_2+\cdots + y_n =0}}_{{y_i\neq 0}}\prod_{i=1}^n \sharp \{a_i x_i^2=y_i\}     = \mathop{\sum_{y_1+y_2+\cdots + y_n =0}}_{{y_i\neq 0}}\prod_{i=1}^n \sum_{\chi_i^2=1}\chi_i \left(\frac{y_i}{a_i}\right)   \ \ \ \ \ \ \ \ \\
&=&\sum_{\chi_1^2=\chi_2^2=\cdots=\chi_n^2=1}\prod_{i=1}^n \chi_i^{-1}(a_i)\mathop{\sum_{y_1+y_2+\cdots + y_n =0}}_{{y_i\neq 0}} \chi_1(y_1)\cdots \chi_n(y_n)  \\
&=& \frac{1}{q}\left[(q-1)^n+(q-1)(-1)^n\right]+\sum_{e=0}^{n-1}\sum_{\mbox{\tiny$\begin{array}{c}
\chi_1^2=\chi_2^2=\cdots=\chi_n^2=1\\
\chi_{i_1}=\cdots =\chi_{i_e}=1\\
\chi_{i_{e+1}}=\cdots =\chi_{i_n}\neq 1\end{array}$}}
\prod_{i=1}^n\chi_i^{-1}(a_i)(-1)^e J_0\left(\chi_{i_{e+1}},\cdots , \chi_{i_n}\right)\\
&=& \frac{1}{q}\left[(q-1)^n+(q-1)(-1)^n\right]+(-1)^n\mathop{\sum_{e=1}^{n}}_{e\ even}\sum_{1\leq i_1<\cdots <i_e\leq n}
\chi(a_{i_1}\cdots a_{i_e}) J_0(\underbrace{\chi,\cdots , \chi}_{e})\\
&=&\frac{1}{q}\left[(q-1)^n+(q-1)(-1)^n\right]+(-1)^n(q-1)\mathop{\sum_{e=1}^n}_{\textrm{$e$ even}}\sum_{i=0}^m(-1)^i\binom{m}{i}\binom{n-m}{e-i}q^{\frac{e}{2}-1}\\
\end{eqnarray*}

Then based on Lemma $3.1$, we have
\begin{eqnarray*}
N^{*}(n,0)&=&\frac{1}{q}\left[(q-1)^n+(q-1)(-1)^n\right]\\
& &+(-1)^n\frac{q-1}{2q}\left[(1-\sqrt{q})^m(1+\sqrt{q})^{n-m}+(1+\sqrt{q})^m(1-\sqrt{q})^{n-m}-2\right]\\
&=&\frac{(q-1)^n}{q}+(-1)^n\frac{q-1}{2q}\left[(1-\sqrt{q})^m(1+\sqrt{q})^{n-m}+(1+\sqrt{q})^m(1-\sqrt{q})^{n-m}\right]
\end{eqnarray*}
Similarly, we can prove the case of $p\equiv 3($mod $4)$ and  $s$ is odd.
\end{proof}

\section{Distinct coordinate sieve}
In this section, we consider the number of solutions with distinct nonzero coordinates of the second moment equations. We denote it by
$$\widetilde{N}^*(k,b)=\sharp\{(x_1,x_2,\cdots,x_k)\in (\mathbb{F}_q^*)^k\mid x_1^2+x_2^2+\cdots+x_k^2=b, x_i\  \textrm{distinct for}\ 1\leq i \leq k\}.$$

In \cite{LiW10}, Li and Wan proposed a new sieve for distinct coordinate counting problem, which greatly improves the classical inclusion-exclusion sieve for this problem.

Let $D$ be a finite set, $D^k=D\times D\times\cdots\times D(k\in\mathbb{N}^+)$ be the cartesian product of $k$ copies of $D$. Let $X$ be a subset of $D^k$.  We are interested in the number of elements in $X$ with distinct coordinates, i.e., the cardinality of the set
$$\overline{X}=\{(x_1,\cdots,x_k)\in X\mid x_i\neq x_j \textrm{ for }\forall \ i\neq j\}.$$

Let $S_k$ be the symmetric group. For a given permutation $\tau\in S_k$, we can write it as disjoint cycle product, i.e., $\tau=(i_1,\cdots,i_{a_1})(j_1,\cdots,j_{a_2})\cdots(l_1,\cdots, l_{a_s})$, where $a_i\geq 1,1\leq i\leq s$. The group $S_k$ acts on $D^k$ by permuting its coordinates, that is
        $$\tau\circ(x_1,\cdots,x_k)=(x_{\tau(1)},\cdots,x_{\tau(k)}).$$
If $X$ is invariant under the action of $S_k$, we call it symmetric. A permutation $\tau\in S_k$ is said to be of type $(c_1,\cdots,c_k)$ if $\tau$ has exactly $c_i$ cycles of length $i$.

In order to illustrate the conclusion, we define
  $$X_\tau=\{(x_1,\cdots,x_k)\mid x_{i_1}=\cdots=x_{ia_1},\cdots,x_{l_1}=\cdots=x_{l_{a_s}}\}.$$
\begin{theorem}\cite{LiW10}
If $X$ is symmetric, we have
\begin{equation}
|\overline{X}|=\sum_{\sum ic_i=k}(-1)^{k-\sum c_i}N(c_1,\cdots,c_k)|X_\tau|,
\end{equation}
where $N(c_1,\cdots,c_k)$ is  the number of permutations in $S_k$ of type $(c_1,\cdots,c_k)$, i.e
$$N(c_1,\cdots,c_k)=\frac{k!}{1^{c_1}c_1!2^{c_2}c_2!\cdots k^{c_k}c_k!}.$$
\end{theorem}

Usually, if $|X_\tau|$ can be written by the form $|X_\tau|=t_1^{c_1}t_2^{c_2}\cdots t_k^{c_k}$ for some nonzero real numbers $t_1, t_2, \cdots, t_k$, we can induce a generating function to compute
$|\overline{X}|$.
\begin{defn}
Define the generating function
$$C_k(t_1,\cdots,t_k)=\sum_{\sum ic_i=k}N(c_1,\cdots,c_k)t_1^{c_1}t_2^{c_2}\cdots t_k^{c_k}.$$
\end{defn}
\begin{lemma}
$(1)$ If $t_1=\cdots = t_k=a$, then
     $$C_k(a,\cdots,a)=(-1)^k(-a)_k,$$
where $(x)_k=x(x-1)\cdots (x-k+1)$ for a real number $x$ and a positive integer $k$.\\
$(2)$ If $t_i=a$ for $p\nmid i$, $t_i=b$ for $p\mid i$, then
       $$C_k(\overbrace{a,\cdots,a}^{p-1},b,\overbrace{a,\cdots,a}^{p-1},b,\cdots)=k!(-1)^k\cdot \sum_{i\geq 0}\binom{-a}{k-pi}\binom{\frac{a-b}{p}}{i}$$
\end{lemma}
\begin{proof}
Firstly, we have the following exponential generating function
    $$\sum_{k\geq 0}C_k(t_1,\cdots,t_k)\frac{u^k}{k!}=e^{ut_1+u^2\cdot\frac{t_2}{2}+u^3\cdot\frac{t_3}{3}+\cdots}.$$
We denote by $\left[\frac{u^k}{k!}\right]f(x)$ the coefficient of $x^i$ in the formal power series expansion of $f(x)$. \\
$(1)$ If $t_1=\cdots = t_k=a$,
\begin{eqnarray*}
C_k(a,\cdots,a)&=& \left[\frac{u^k}{k!}\right]e^{a\left(u+\frac{u^2}{2}+\frac{u^3}{3}+\cdots\right)}= \left[\frac{u^k}{k!}\right]e^{-a\log(1-u)}\\
&=&\left[\frac{u^k}{k!}\right](1-u)^{-a}=\left[\frac{u^k}{k!}\right]\sum_{i\geq 0}(-1)^i\binom{-a}{i}u^i\\
&=&(-1)^k k!\binom{-a}{k}=(-1)^k(-a)_k.
\end{eqnarray*}
$(2)$  If $t_i=a$ for $p\nmid i$, $t_i=b$ for $p\mid i$, then
\begin{eqnarray*}
 &&C_k(\overbrace{a,\cdots,a}^{p-1},b,\overbrace{a,\cdots,a}^{p-1},b,\cdots) =\left[\frac{u^k}{k!}\right]e^{ua+u^2\cdot\frac{a}{2}+\cdots+u^{p-1}\frac{a}{p-1}+u^p\cdot\frac{b}{p}+u^{p+1}\frac{a}{p+1}\cdots}\\
&=& \left[\frac{u^k}{k!}\right]e^{-a\log(1-u)-\frac{b-a}{p}\log(1-u^p)}=\left[\frac{u^k}{k!}\right](1-u)^{-a}\cdot(1-u^p)^\frac{a-b}{p}\\
&=&\left[\frac{u^k}{k!}\right]\sum_{j\geq 0}(-1)^j\binom{-a}{j}u^j\sum_{i\geq 0}(-1)^i\binom{\frac{a-b}{p}}{i}u^{pi}=k!(-1)^k\cdot \sum_{i\geq 0}\binom{-a}{k-pi}\binom{\frac{a-b}{p}}{i}.
\end{eqnarray*}
\end{proof}

\begin{defn}
For $b\in \mathbb{F}_q^*$, we define
$$A_{k,b}(u,v,w)=C_k(t_1,\cdots,t_k),$$
where
$$t_i=\left\{
\begin{array}{ll}  u & p\nmid i, \chi(i)=\chi(b),\\ v & p\nmid i, \chi(i)=-\chi(b),\\ w & p\mid i.
\end{array}\right.$$
Note that $A_{k,b}(u,v,w)$ depends only on the value of $\chi(b)$, not on $b$.
In the case $u=v\neq w$, we simply write $A_{k,b}(u,v,w)$ as $A_{k,b}(u,w)$ and have the greatly simplifed formula
$$A_{k,b}(u,w)=k!(-1)^k\cdot \sum_{i\geq 0}\binom{-u}{k-pi}\binom{\frac{u-w}{p}}{i}.$$
\end{defn}
\begin{theorem}
Let $\mathbb{F}_{q}$ be a finite field with $q=p^s$.  For all $b\in \mathbb{F}_q^*$,\\
$(1)$ If either  $p\equiv 1($mod $4)$, or $p\equiv 3($mod $4)$ and $s$ is even, we have
 \begin{eqnarray*}
 \widetilde{N}^*(k,b)&=&\frac{(q-1)_k}{q}-\frac{(-1)^k}{2q}\{(1-\sqrt{q})A_{k,b}(1-\sqrt{q},1+\sqrt{q},1-q)\\
    & &+(1+\sqrt{q})A_{k,b}(1+\sqrt{q},1-\sqrt{q},1-q)\}.
 \end{eqnarray*}
 $(2)$ If  $p\equiv 3($mod $4)$, $s$ is odd, we have
 \begin{eqnarray*}
 \widetilde{N}^*(k,b)&=&\frac{(q-1)_k}{q}-\frac{(-1)^k}{2q}\{(1+\sqrt{q}i)A_{k,b}(1-\sqrt{q}i,1+\sqrt{q}i,1-q)\\
  & &  +(1-\sqrt{q}i)A_{k,b}(1+\sqrt{q}i,1-\sqrt{q}i,1-q)\},
 \end{eqnarray*}
 where $i=\sqrt{-1}$.
\end{theorem}
\begin{proof}
   Let $X^*=\{(x_1,\cdots,x_k)\in (\mathbb{F}_q^*)^k\mid x_1^2+x_2^2+\cdots+x_k^2=b\}$. Obviously, $X^*$ is symmetric, so we can use the relation
   $$\widetilde{N}^*(k,b)=\sum_{\sum ic_i=k}(-1)^{k-\sum c_i}N(c_1,\cdots,c_k)|X_\tau^*|,$$
where
   $\tau\in S_k$ is  of type $(c_1,\cdots,c_k)$, $X_\tau^*=\{(x_{11},\cdots,x_{kc_k})\in \left(\mathbb{F}_q^*\right)^{\sum c_i}\mid x_{11}^2+\cdots+x_{1c_1}^2+2x_{21}^2+\cdots+2x_{2c_2}^2+\cdots+kx_{k1}^2+\cdots+kx_{kc_k}^2=b\}.$ Then we need to compute $|X_\tau^*|$ first.

 Set $m=\sum_{p\mid i}c_i$, $n=\sum_{p\nmid i}c_i$, $m_0=\sum_{1\leq i \leq k, p\nmid i,\chi(i)=\chi(b)}c_i$. Obviously, $n+m=\sum_i c_i$, $n-m_0=\sum_{1\leq i \leq k, p\nmid i, \chi(i)=-\chi(b)}c_i$, and the number of the variables in the equation $ x_{11}^2+\cdots+x_{1c_1}^2+2x_{21}^2+\cdots+2x_{2c_2}^2+\cdots+kx_{k1}^2+\cdots+kx_{kc_k}^2=b$ is $n$.  If $n\neq 0$, we can use the conclusion of theorem $3.3$ to compute $|X_{\tau}^*|$.

If either $p\equiv 1($mod $4)$, or $p\equiv 3($mod $4)$ and $s$ is even,
\begin{eqnarray*}
   && |X_{\tau}^*|=(q-1)^{m} \sharp \{(\cdots,x_{it_i},\cdots)\in (\mathbb{F}_q^*)^n\mid \sum_{\mbox{\tiny$\begin{array}{c}
1\leq i \leq k, 1\leq t_i\leq c_i\\
p\nmid i\end{array}$}}ix_{it_i}^2=b\}\\
&=& (q-1)^{m}\{\frac{(q-1)^n}{q}-\frac{(-1)^n}{2q}\left[\left(1-\sqrt{q}\right)^{m_0+1}\left(1+\sqrt{q}\right)^{n-m_0}+
 \left(1+\sqrt{q}\right)^{m_0+1}\left(1-\sqrt{q}\right)^{n-m_0}\right]\}\
\end{eqnarray*}

Otherwise, if $n=0$, i.e., $\tau$ consists only of cycles with length divisible by $p$, then the equation $ x_{11}^2+\cdots+x_{1c_1}^2+2x_{21}^2+\cdots+2x_{2c_2}^2+\cdots+kx_{k1}^2+\cdots+kx_{kc_k}^2=b$ changes into $0=b$, which never holds for $b\neq 0$. So in this case $|X_{\tau}^*|=0$. Fortunately, this result is formally in accordance with the case of $n\neq 0$.

    So we have
 \begin{eqnarray*}
    \widetilde{N}^*(k,b)&=&\sum_{\sum ic_i=k}(-1)^{k-\sum c_i}N(c_1,\cdots,c_k)|X_\tau^*|\\
    &=&(-1)^k\sum_{\sum ic_i=k}N(c_1,\cdots,c_k)(-1)^{\sum c_i}|X_\tau^*|\\
    &=&\frac{(q-1)_k}{q}-\frac{(-1)^k}{2q}\{(1-\sqrt{q})A_{k,b}(1-\sqrt{q},1+\sqrt{q},1-q)\\
    & &+(1+\sqrt{q})A_{k,b}(1+\sqrt{q},1-\sqrt{q},1-q)\}
 \end{eqnarray*}
 
 The case of $p\equiv 3($mod $4)$ and $s$ is odd can be proved similarly.
 \end{proof}

\begin{theorem} Let $\mathbb{F}_{q}$ be a finite field with $q=p^s$.\\
$(1)$ If either $p\equiv 1($mod $4)$,  or $p\equiv 3($mod $4)$ and $s$ is even,  we have
  $$\widetilde{N}^*(k,0)=\frac{(q-1)_k}{q}+(-1)^k\frac{q-1}{2q}\{A_{k,1}(1-\sqrt{q},1+\sqrt{q},1-q)+A_{k,1}(1+\sqrt{q},1-\sqrt{q},1-q)\}.$$
   $(2)$ If  $p\equiv 3($mod $4)$, $s$ is odd, we have
 \begin{eqnarray*}
 \widetilde{N}^*(k,0)&=&\frac{(q-1)_k}{q}+(-1)^k\frac{q-1}{2q}\{A_{k,1}(1-\sqrt{q}i,1+\sqrt{q}i,1-q)\\
    & &+A_{k,1}(1+\sqrt{q}i,1-\sqrt{q}i,1-q)\}.
 \end{eqnarray*}
 \end{theorem}
 \begin{proof}
Let $X_0^*=\{(x_1,\cdots,x_k)\in (\mathbb{F}_{q}^*)^k\mid x_1^2+\cdots+x_k^2=0\}$. Since $X_0^*$ is symmetric, we can apply
$$N^*(k,0)=\sum_{\sum ic_i=k}(-1)^{k-\sum c_i}N(c_1,\cdots,c_k)|X_{0\tau}^*|,$$
where $\tau$ is of type $(c_1,\cdots, c_k)$, $X_{0\tau}^*=\{(x_{11},\cdots,x_{kc_k})\in \left(\mathbb{F}_q^*\right)^{\sum c_i}\mid x_{11}^2+\cdots+x_{1c_1}^2+2x_{21}^2+\cdots+2x_{2c_2}^2+\cdots+kx_{k1}^2+\cdots+kx_{kc_k}^2=0\}.$ Next we need to compute $|X_{0\tau}^*|$.
 Set $m=\sum_{p\mid i}c_i$, $n=\sum_{p\nmid i}c_i$, $m_0=\sum_{1\leq i \leq k, p\nmid i, \chi(i)=1}c_i$. Obviously, $n+m=\sum_i c_i$, $n-m_0=\sum_{1\leq i \leq k, p\nmid i, \chi(i)=-1}c_i$, and the number of the variables in the equation
 $x_{11}^2+\cdots+x_{1c_1}^2+2x_{21}^2+\cdots+2x_{2c_2}^2+\cdots+kx_{k1}^2+\cdots+kx_{kc_k}^2=0$ is $n$. If $n\neq 0$, we can use the conclusion of theorem $3.5$ to compute $|X_{0\tau}^*|$.

 If either $p\equiv 1($mod $4)$,  or $p\equiv 3($mod $4)$ and $s$ is even,
\begin{eqnarray*}
    &&|X_{0\tau}^*|=(q-1)^{m} \sharp \{(\cdots,x_{it_i},\cdots)\in (\mathbb{F}_q^*)^n\mid \sum_{\mbox{\tiny$\begin{array}{c}
1\leq i \leq k, 1\leq t_i\leq c_i\\
p\nmid i\end{array}$}}ix_{it_i}^2=0\}\\
    &=&(q-1)^m\{\frac{(q-1)^n}{q}+(-1)^n\frac{q-1}{2q}\left[(1-\sqrt{q})^{m_0}(1+\sqrt{q})^{n-m_0}+(1+\sqrt{q})^{m_0}(1-\sqrt{q})^{n-m_0}\right]\}
\end{eqnarray*}
Otherwise, if $n=0$, i.e., $\tau$ consists only of cycles with length divisible by $p$, the equation $ x_{11}^2+\cdots+x_{1c_1}^2+2x_{21}^2+\cdots+2x_{2c_2}^2+\cdots+kx_{k1}^2+\cdots+kx_{kc_k}^2=0$ changes into $0=0$, which holds for any $x_{it_i}\in \mathbb{F}_{q}^*(1\leq i\leq n)$. So in this case $|X_{\tau}^*|=(q-1)^{\sum_i c_i}$. Fortunately, this result also formally coincides with the case of $n\neq 0$.

   So we have
 \begin{eqnarray*}
    & &\widetilde{N}^*(k,0)=(-1)^k\sum_{\sum ic_i=k}N(c_1,\cdots,c_k)(-1)^{\sum c_i}|X_{0\tau}^*|\\
       &=&\frac{(q-1)_k}{q}+(-1)^k\frac{q-1}{2q}\{A_{k,1}(1-\sqrt{q},1+\sqrt{q},1-q)+A_{k,1}(1+\sqrt{q},1-\sqrt{q},1-q)\}
 \end{eqnarray*}
Note that we have $A_{k,1}(\cdot, \cdot, \cdot)$ in the above equation because of the definition of $m_0$ and $A_{k,b}(\cdot, \cdot, \cdot)$ for $b\neq 0$.  The other case can be proved similarly.
\end{proof}

\section{Sieve again for solving the subset sum problem}
Now let us come back to the subset sum problem on quadratic residues. Since $H=\{x^2\mid x\in \mathbb{F}_{q}^*\}$, it is obviously a subgroup of $\mathbb{F}_{q}^*$ with $\frac{q-1}{2}$ elements, and every element $y\in H$ can be expressed as $y=x^2$ for exactly two $x\in \mathbb{F}_{q}^*$. What's more,  $x_i^2=x_j^2$ iff $x_i=\pm x_j$. Therefore, if we denote by
$$Y=\{(y_1, y_2, \cdots, y_k)\in H^k\mid y_1+y_2+\cdots +y_k=b\},$$
 $$\widetilde{Y}=\{(y_1, y_2, \cdots, y_k)\in H^k\mid y_1+y_2+\cdots +y_k=b, y_i\neq y_j\textrm{ for }\forall \ \ i\neq j\}.$$
  Then obviously, $N_H(k,b)=\frac{1}{k!}|\widetilde{Y}|$, and  we can also use $(4.1)$ to calculate $|\widetilde{Y}|$ since $Y$ is symmetric. As a consequence, we have
  \begin{equation}
  N_H(k,b)=\frac{1}{k!}\sum_{\sum ic_i=k}(-1)^{k-\sum c_i}N(c_1,\cdots,c_k)|\widetilde{Y}_\tau|,
  \end{equation}
  where $\tau$ is a permutation of type $(c_1,c_2,\cdots,c_k)$ in $S_k$, and
 $$\widetilde{Y}_\tau=\{(y_{11},\cdots,y_{1c_1},\cdots,y_{kc_k})\in H^{\sum c_i}\mid y_{11}+\cdots+y_{1c_1}+\cdots+ky_{kc_k}=b\}.$$
On the other hand, we have defined in the previous section that
$$X_\tau^*=\{(x_{11},\cdots,x_{1c_1},\cdots,x_{kc_k})\in(\mathbb{F}_{q}^*)^{{\sum c_i}}\mid x_{11}^2+\cdots+x_{1c_1}^2+\cdots+kx_{kc_k}^2=b\},$$
and have calculated the cardinality of $X_\tau^*$. It is also not hard to see
\begin{equation}
|\widetilde{Y}_\tau|=\frac{|X_\tau^*|}{2^{\sum c_i}}.
\end{equation}
 Replacing  $|\widetilde{Y}_\tau|$ in $(5.1)$ by $\frac{|X_\tau^*|}{2^{\sum c_i}}$ yields the following conclusion.

\begin{theorem}Let $\mathbb{F}_{q}$ be a finite field with $q=p^s$. For any $b\in \mathbb{F}_{q}^*$, \\
$(1)$ If either $p\equiv 1($mod $4)$, or $p\equiv 3($mod $4)$ and  $s$ is even, we have
 \begin{eqnarray*}
 N_H(k,b)&=&\frac{1}{q}\binom{\frac{q-1}{2}}{k}-\frac{(-1)^k}{2qk!}\{(1-\sqrt{q})A_{k,b}(\frac{1-\sqrt{q}}{2},\frac{1+\sqrt{q}}{2},\frac{1-q}{2})\\
    & &+(1+\sqrt{q})A_{k,b}(\frac{1+\sqrt{q}}{2},\frac{1-\sqrt{q}}{2},\frac{1-q}{2})\}.
 \end{eqnarray*}
 $(2)$ If  $p\equiv 3($mod $4)$, $s$ is odd, we have
 \begin{eqnarray*}
N_H(k,b)&=&\frac{1}{q}\binom{\frac{q-1}{2}}{k}-\frac{(-1)^k}{2qk!}\{(1+\sqrt{q}i)A_{k,b}(\frac{1-\sqrt{q}i}{2},\frac{1+\sqrt{q}i}{2},\frac{1-q}{2})\\
    & &+(1-\sqrt{q}i)A_{k,b}(\frac{1+\sqrt{q}i}{2},\frac{1-\sqrt{q}i}{2},\frac{1-q}{2})\},
 \end{eqnarray*}
where $i=\sqrt{-1}$.
\end{theorem}

Similarly, we can derive the following results of subset sums if $b=0$.
\begin{theorem}Let $\mathbb{F}_{q}$ be a finite field with $q=p^s$.\\
$(1)$ If either $p\equiv 1($mod $4)$, or  $p\equiv 3($mod $4)$ and  $s$ is even,we have
 \begin{eqnarray*}
 N_H(k,0)&=&\frac{1}{q}\binom{\frac{q-1}{2}}{k}+(-1)^k\frac{q-1}{2qk!}\{A_{k,1}(\frac{1-\sqrt{q}}{2},\frac{1+\sqrt{q}}{2},\frac{1-q}{2})\\
    & &+A_{k,1}(\frac{1+\sqrt{q}}{2},\frac{1-\sqrt{q}}{2},\frac{1-q}{2})\}.
 \end{eqnarray*}
 $(2)$ If  $p\equiv 3($mod $4)$, $s$ is odd, we have
 \begin{eqnarray*}
N_H(k,0)&=&\frac{1}{q}\binom{\frac{q-1}{2}}{k}+(-1)^k\frac{q-1}{2qk!}\{A_{k,1}(\frac{1-\sqrt{q}i}{2},\frac{1+\sqrt{q}i}{2},\frac{1-q}{2})\\
    & &+A_{k,1}(\frac{1+\sqrt{q}i}{2},\frac{1-\sqrt{q}i}{2},\frac{1-q}{2})\}.
 \end{eqnarray*}
where $i=\sqrt{-1}$.
 \end{theorem}
Note that all the results we discussed in the above sections are all on finite fields  $\mathbb{F}_q$  with $p^s$ elements and  odd characteristic $p$.  Especially, if $s$ is even, our computational formula can be greatly simplified because of the following proposition.

\begin{prop}
Let $\mathbb{F}_q$ be a finite field with odd characteristic $p$, $q=p^s$. Let $\chi$ be the quadratic character of $\mathbb{F}_q^*$
and $\psi$ be the quadratic character of $\mathbb{F}_p^*$, then the restriction of $\chi$ to $\mathbb{F}_p^*$ is
$\psi^s$, which is trivial if $s$ is even.
\end{prop}

In this occasion, we can derive more concise explicit formulas for $N_H(k, b)$.
\begin{cor}
Let $\mathbb{F}_{q}$ be a finite field with $q=p^s$, where $p$ is an odd prime and $s$ is even. For any $b\in \mathbb{F}_{q}^*$, we have
 \begin{eqnarray*}
& &N_H(k, b)=\frac{1}{q}\binom{\frac{q-1}{2}}{k}\\
&-&\left\{
\begin{array}{ll}\frac{(-1)^k}{2qk!}\{(1-\sqrt{q})A_{k,b}(\frac{1-\sqrt{q}}{2},\frac{1-q}{2})-
(1+\sqrt{q})A_{k,b}(\frac{1+\sqrt{q}}{2},\frac{1-q}{2})\} & \textrm{if $\chi(b) = 1$,} \\ \frac{(-1)^k}{2qk!}\{(1-\sqrt{q})A_{k,b}(\frac{1+\sqrt{q}}{2},\frac{1-q}{2})-
(1+\sqrt{q})A_{k,b}(\frac{1-\sqrt{q}}{2},\frac{1-q}{2})\}  & \textrm{if $\chi(b) = -1$.}
\end{array}\right.
 \end{eqnarray*}
 \end{cor}

\begin{cor}
Let $\mathbb{F}_{q}$ be a finite field with $q=p^s$, where $p$ is an odd prime and $s$ is even.  We have\\
$$ N_H(k, 0)=\frac{1}{q}\binom{\frac{q-1}{2}}{k}+(-1)^k\frac{q-1}{2qk!}\{A_{k,1}(\frac{1+\sqrt{q}}{2},\frac{1-q}{2}) +A_{k,1}(\frac{1-\sqrt{q}}{2},\frac{1-q}{2})\}.$$
\end{cor}

\end{document}